\documentclass[oneside,a4paper]{amsart}

\usepackage[T1]{fontenc}
\usepackage[utf8]{inputenc}
\usepackage[british]{babel}
\usepackage{amsmath, amsthm, amsfonts, amssymb}

\usepackage{url}
\usepackage[all]{xy}
\usepackage{placeins} 
\usepackage[usenames, dvipsnames]{color} 
\usepackage{euscript} 
\usepackage{bbm} 
\usepackage{enumitem} 

\usepackage[dvipsnames]{xcolor} 
\definecolor{dark-red}{rgb}{0.5,0.15,0.15}
\definecolor{dark-blue}{rgb}{0.15,0.15,0.6}
\definecolor{dark-green}{rgb}{0.15,0.6,0.15}
\usepackage{hyperref}
\hypersetup{
    colorlinks, linkcolor=OliveGreen,
    citecolor=dark-blue, urlcolor=dark-green
}
\usepackage[nameinlink,capitalise,noabbrev]{cleveref}

\usepackage{tikz}
\usepackage{tikz-cd}
\usetikzlibrary{arrows,backgrounds,
    decorations.pathreplacing,
    decorations.pathmorphing
}
\usetikzlibrary{matrix,arrows}

\newcommand{\euscr}[1]{\EuScript{#1}} 
\newcommand{\ccat}{\euscr{C}} 
\newcommand{\dcat}{\euscr{D}} 

\newcommand{\Ext}{\textnormal{Ext}} 
\newcommand{\map}{\textnormal{map}} 
\newcommand{\mfrak}{\mathfrak{m}} 
\newcommand{\spectra}{\euscr{S}p} 

\newcommand{\ComodE}{\euscr{C}omod_{E_{*}E}} 
\newcommand{\Mod}{\euscr{M}od} 
\newcommand{\monunit}{\mathbbm{1}} 
\newcommand{\Pic}{\mathrm{Pic}} 
\newcommand{\Picspace}{\mathcal{P}\mathrm{ic}}

\theoremstyle{plain}

\newtheorem{theorem}{Theorem}[section]
\newtheorem{lemma}[theorem]{Lemma}
\newtheorem{prop}[theorem]{Proposition}

\newtheorem*{theorem*}{Theorem}

\theoremstyle{definition}

\newtheorem{rem}[theorem]{Remark}

\newtheorem*{rem*}{Remark}
\newtheorem*{interpretation*}{Interpretation}
\newtheorem*{defin*}{Definition}
\newtheorem*{conjecture*}{Conjecture}
\newtheorem*{notation*}{Notation}
\newtheorem*{convention*}{Convention}
\newtheorem*{theorem_italics*}{Theorem}

\theoremstyle{remark}

\makeatletter
  \def\subsection{\@startsection{subsection}{1}%
  \z@{.7\linespacing\@plus\linespacing}{.5\linespacing}%
  {\normalfont\bfseries\centering}}
\makeatother

\calclayout

\begin{document}

\title[Chromatic Picard groups at large primes]{Chromatic Picard groups at large primes}
\author[Piotr Pstr\k{a}gowski]{Piotr Pstr\k{a}gowski}
\address{Harvard University}
\email{pstragowski.piotr@gmail.com}

\begin{abstract}
We show that the Hopkins' Picard group of the $K(n)$-local category coincides with its algebraic approximation when $2p-2 > n^{2}+n$.
\end{abstract}

\maketitle 

\section{Introduction}

If $\ccat$ is a symmetric monoidal $\infty$-category, then we can consider equivalence classes of invertible objects $X$, that is, those such that there exists a $Y$ satisfying $X \otimes Y \simeq \monunit$. This is often a set, rather than a proper class, and it inherits a group multiplication induced from the tensor product. We call the resulting group the Picard group and denote it by $\Pic(\ccat)$. 

Following ideas of Hopkins, the study of Picard groups was brought into chromatic homotopy theory \cite{hopkins1994constructions}, \cite{strickland1992p}. In this context, $\ccat$ is usually taken to be the $\infty$-category of $E(n)$- or $K(n)$-local spectra at a fixed prime. 

 As a general rule, one expects the answers to be algebraic when the prime is large compared to the height. To explain what we mean, let us focus on the $E(n)$-local case first. In this context, taking rational homology defines a homomorphism 
 
 \begin{center}
 $H\mathbb{Q}_{*}: \Pic(\spectra_{E(n)}) \rightarrow \Pic(\mathbb{Q})$, 
 \end{center}
 where by the latter we denote the Picard group of graded rational vector spaces, which is isomorphic to $\mathbb{Z}$. This homomorphism is in fact a split surjection, with splitting $k \mapsto S^{k}_{E(n)}$. 
 
Then, it is a result of Hovey and Sadofsky that when $2p-2 > n^{2}+n$, the algebraic comparison map is an isomorphism, so that we have $\Pic(\spectra_{E(n)}) \simeq \mathbb{Z}$ \cite{hovey1999invertible}. This is in stark contrast with what happens at small primes; for example, we have $\textnormal{Pic}(\spectra_{E(1)}) \simeq \mathbb{Z} \oplus \mathbb{Z}/2$ at $p = 2$, and $\textnormal{Pic}(\spectra_{E(2)}) \simeq \mathbb{Z} \oplus \mathbb{Z}/3 \oplus \mathbb{Z}/3$ at $p = 3$ \cite{hovey1999invertible}, \cite{goerss2014hopkins}. 

To study the $K(n)$-local case, one needs a more subtle algebraic invariant. More precisely, we define the completed $E$-homology as 

\begin{center}
$E_{*}^{\vee} X := \pi_{*}L_{K(n)} (E \wedge X)$,
\end{center}
where $E$ is the Morava $E$-theory spectrum of height $n$. When it's finitely generated, $E_{*}^{\vee} X$ has a canonical structure of an $L$-complete comodule over $E_{*}^{\vee} E$ \cite{baker2009complete} \cite{barthel2016e2}[1.22]. The latter Hopf algebroid can be described explicitly as $E^{\vee}_{*}E \simeq \map^{c}(\mathbb{G}_{n}, E_{*})$, the space of continuous functions on the Morava stabilizer group, with structure maps induced from the action of $\mathbb{G}_{n}$ \cite{devinatz2004homotopy}.  

If $X$ is $K(n)$-locally invertible, then $E_{*}^{\vee} X$ is an invertible $E_{*}^{\vee}E$-comodule, which gives a homomorphism $\Pic(\spectra_{K(n)}) \rightarrow \Pic(E_{*}^{\vee} E)$ into the algebraic Picard group, given by isomorphisms classes of invertible comodules. 

The algebraic Picard group can be expressed in terms of cohomology of the Morava stabilizer group; to do so, one observes that an invertible $E_{*}^{\vee}E$-comodule is the same as an invertible $E_{*}$-module equipped with a compatible continuous action of $\mathbb{G}_{n}$. Since $E_{0}$ is a regular local ring, any such module is free of rank one, and so we have a short exact sequence 

\begin{center}
$0 \rightarrow \Pic^{0}(E_{*}^{\vee}E) \rightarrow \Pic(E_{*}^{\vee}E) \rightarrow \mathbb{Z}/2 \rightarrow 0$,
\end{center}
where $\Pic^{0}(E_{*}^{\vee}E)$ is the subgroup of those invertible modules which are concentrated in even degrees. Since $E_{*}$ is $2$-periodic, any such module is determined by its degree zero part, which yields an isomorphism $\Pic^{0}(E_{*}^{\vee}E) \simeq H_{c}^{1}(\mathbb{G}_{n}, E_{0}^{\times})$ by standard considerations \cite{goerss2014hopkins}. 

Due to a classical argument using the sparsity of the Adams-Novikov spectral sequence, one knows that $\Pic(\spectra_{K(n)}) \rightarrow \Pic(E_{*}^{\vee} E)$ is injective when $2p-2 > n^{2}$ and $(p-1) \nmid n$ \cite{hopkins1994constructions}[7.5]. On the other hand, surjectivity was not known except at low heights, where both sides can be computed explicitly. 

Our main result gives a range in which the comparison map is in fact an isomorphism. 

\begin{theorem}[\ref{thm:picard_groups_of_k_local_cat_algebraic_at_large_primes}]
\label{thm:introduction:topological_and_algebraic_picard_groups_isomorphic_for_2pminus2_larger_than_nsquared_n}
When $2p-2 > n^{2}+n$, $\Pic(\spectra_{K(n)}) \rightarrow \Pic(E_{*}^{\vee} E)$ is an isomorphism.
\end{theorem}
The proof of \cref{thm:introduction:topological_and_algebraic_picard_groups_isomorphic_for_2pminus2_larger_than_nsquared_n} rests on the recent chromatic algebraicity result of the author which states that at large primes, there exists an equivalence  $h \spectra _{E} \simeq h \dcat(E_{*}E)$ between the homotopy categories of $E$-local spectra and differential $E_{*}E$-comodules \cite{pstragowski_chromatic_homotopy_algebraic}.

In fact, to prove the isomorphism between Picard groups, we do not need the equivalence of homotopy categories, but only the weaker statement that any $E_{*}E$-comodule can be canonically realized as a homology of a certain $E$-local spectrum. Thus, \cref{thm:introduction:topological_and_algebraic_picard_groups_isomorphic_for_2pminus2_larger_than_nsquared_n} holds in a slightly larger range of primes than chromatic algebraicity. 

The arguments we use are quite general, and we believe could be applied in many other profinite contexts. In particular, a generalization to the case of $K(n)$-locally dualizable spectra appears in the work of Barthel, Heard and Naumann \cite{barthel2020conjectures}.

As was pointed to us by Paul Goerss, a more natural proof of \cref{thm:introduction:topological_and_algebraic_picard_groups_isomorphic_for_2pminus2_larger_than_nsquared_n} would use the descent spectral sequence for $S^{0}_{K(n)} \rightarrow E$, which is $K(n)$-local pro-Galois extension with Galois group $\mathbb{G}_{n}$ \cite{rognes2005galois}. No such spectral sequence was known at the time this article first appeared, but it has been since then constructed by Heard \cite[\S 6C]{heard2021spkn}. 

Since it is of potential interest, we present this alternative approach in \cref{rem:collapse_of_descent_spectral_sequence_at_large_primes}. The proof given in the main body of the article is independent from this argument.

\subsection{Acknowledgements}

I would like to thank my supervisor Paul Goerss for his support and guidance, as well as for helpful comments on the structure of this paper. 

\section{Chromatic Picard groups at large primes}

We let $p$ denote the prime and $n$ the height, both of which are fixed. By $E$ we denote the Morava $E$-theory spectrum, this is an even periodic Landweber exact spectrum associated to the Lubin-Tate ring $E_{0} \simeq W(\mathbb{F}_{p^{n}})[[u_{1}, \ldots, u_{n-1}]]$ of the Honda formal group law over $\mathbb{F}_{p^{n}}$, see for example \cite[\S1]{goerss2005resolution} for more on the setup. In particular, $E_{0}$ is a complete regular local ring of dimension $n$, with maximal ideal $\mfrak = (p, u_{1}, \ldots, u_{n-1})$.

The completion functor $M \rightarrow \varprojlim M / \mfrak^{k} M$ on $E_{*}$-modules is neither right or left exact, but it has a right exact left derived functor which we denote by $L_{0}$ \cite[Appendix A]{hovey1999morava}. We say a module $M$ is $L$-complete if the natural map $M \rightarrow L_{0}M$ is an isomorphism. If $M, N$ are modules, then we denote their $L$-complete tensor product by $M \widehat{\otimes} _{E_{*}} N := L_{0}(M \otimes _{E_{*}} N)$. 

We let $K$ denote the associated Morava $K$-theory spectrum; this is the unique up to equivalence $E$-module which admits an $E$-algebra structure whose unit induces an isomorphism $K_{*} \simeq E_{*} / \mfrak$ \cite{lurie_hopkins_brauer_group}. The spectrum $K$ is Bousfield equivalent to the classical Morava $K$-theory spectrum $K(n)$ satisfying $K(n)_{*} \simeq \mathbb{F}_{p}[v_{n}^{\pm 1}]$. 

Following \cite{hopkins1994constructions}, \cite{strickland1992p}, for a spectrum $X$ we define its completed $E$-homology as 

\begin{center}
$E_{*}^{\vee} X := \pi_{*}L_{K} (E \wedge X)$.
\end{center}
One can show that $E_{*}^{\vee}X$ is an $L$-complete $E_{*}$-module, and if it is finitely generated, then it has a structure of a comodule over $E_{*}^{\vee}E$ \cite[8.5]{hovey1999morava}, \cite{baker2009complete}, \cite[1.22]{barthel2016e2}. If $X$ is finite, or more generally if $E_{*}X$ is $L$-complete, then $E^{\vee}_{*}X \simeq E_{*}X$  \cite[3.2]{hovey2004some}.

We would like to understand the Picard group $\Pic(\spectra_{K})$ of equivalence classes of invertible $K$-local spectra. We begin by recalling the following fundamental result. 

\begin{theorem}[\cite{hopkins1994constructions}(1.3)]
\label{thm:invertibility_detected_by_completed_e_homology}
A spectrum $X$ is $K$-locally invertible if and only if $E_{*}^{\vee }X$ is free of rank one over $E_{*}$; equivalently, is an invertible $E_{*}^{\vee}E$-comodule. 
\end{theorem}
As a consequence of \cref{thm:invertibility_detected_by_completed_e_homology}, we obtain a homomorphism $E_{*}^{\vee}: \Pic(\spectra_{K}) \rightarrow Pic(E_{*}^\vee E)$ from the $K$-local Picard group into the Picard group of $E_{*}^{\vee}E$, given by isomorphisms classes of $E_{*}^{\vee}E$-comodules which are invertible under the tensor product. One can describe $E_{*}^{\vee} E$ and the associated Picard group in terms of the Morava stabilizer group, which we now recall. 

Since $E$ is even periodic, it is complex orientable and the associated formal group is the universal deformation of the Honda formal group law $\Gamma$ of height $n$ over $\mathbb{F}_{p^{n}}$. This endows $E_{0}$ with an action of the Morava stabilizer group $\mathbb{G}_{n} := \textnormal{Aut}(\mathbb{F}_{p^{n}}, \Gamma) \simeq \textnormal{Aut}(\Gamma) \rtimes \textnormal{Gal}(\mathbb{F}_{p^{n}} / \mathbb{F}_{p})$, which by the Goerss-Hopkins-Miller theorem lifts to an action on $E$ by maps of commutative ring spectra \cite{moduli_problems_for_structured_ring_spectra},  \cite{pstrkagowski2021abstract}. 

The action of $\mathbb{G}_{n}$ on $E$ induces an isomorphism $E_{*}^{\vee}E \simeq \map_{c}(\mathbb{G}_{n}, E_{*})$, where the latter is the space of continuous functions on the Morava stabilizer group \cite{devinatz2004homotopy}. If $M$ is an $E_{*}^{\vee} E$-comodule, then this identification endows it with an action of $\mathbb{G}_{n}$, and if $M$ is finitely generated over $E_{*}$, then this action is continuous in the $\mfrak$-adic topology and any such continuous action determines a comodule structure \cite[5.4]{barthel2016e2}. 

We deduce that the data of an invertible $E_{*}^{\vee} E$-comodule is the same as that of a an invertible $E_{*}$-module equipped with a compatible continuous action of $\mathbb{G}_{n}$, this allows one to give a homological description of $\Pic(E_{*}^{\vee} E)$, as we recalled in the introduction.

Our goal is to prove that the homomorphism $\Pic(\spectra_{K}) \rightarrow \Pic(E_{*}^{\vee} E)$ is an isomorphism at large primes. We start with injectivity, which is classical, but since the proof is enlightening, and not particularly difficult, we briefly recall the argument. 

\begin{prop}[\cite{hopkins1994constructions}(7.5)]
\label{prop:comparison_map_injective_at_large_primes}
If $2p-2 \geq n^{2}$ and $(p-1) \nmid n$, then the comparison map $E_{*}^{\vee}: \Pic(\spectra_{K}) \rightarrow \Pic(E_{*}^{\vee} E)$ is injective. 
\end{prop}

\begin{proof}
Suppose that $X \in \Pic(\spectra_{K})$; since $E_{*}^{\vee} X$ is free of rank one, in particular finitely generated, we have the $K$-local $E$-based Adams spectral sequence of the form
\begin{center}
$\widehat{\Ext}^{s, t}_{E_{*}^{\vee} E}(E_{*}, E_{*}^{\vee} X) \Rightarrow \pi_{t-s} X$
\end{center}
and an isomorphism $\Ext^{s,t}_{E_{*}^{\vee}E}(E_{*}, E_{*}^{\vee}X) \simeq H_{c}^{s}(\mathbb{G}_{n}, E_{t}^{\vee} X)$ between the $E_{2}$-term and the continuous cohomology of the Morava stabilizer group \cite[3.1, 4.1]{barthel2016e2}. 

The $E_{2}$-term is concentrated in internal degrees divisible by $2p-2$ and if $(p-1) \nmid n$, then it has a horizontal vanishing line at $n^{2}$, the homological dimension of the Morava stabilizer group \cite[4.2.1]{heard2015morava}. It follows that under the given assumptions the spectral sequence collapses for degree reasons.

Now, suppose that $X$ is in the kernel of $E_{*}^{\vee}: \Pic(\spectra_{K}) \rightarrow \Pic(E_{*}^{\vee}E)$, so that we have an isomorphism $E_{*}^{\vee} X \simeq E_{*}^{\vee} S_{K}^{0} \simeq E_{*}$. As observed above, the $E$-based Adams spectral sequence collapses, and it follows that the chosen isomorphism is necessarily an infinite cycle and so descends to an equivalence $X \simeq S^{0}_{K}$. This ends the argument. 
\end{proof}
We move on to the surjectivity of the comparison map; this is the heart of the problem. We start with two short, technical lemmas.

\begin{lemma}
\label{lemma:local_homology_of_residue_field}
We have $\varinjlim \Ext_{E_{*}}(E_{*} / \mfrak^{k}, K_{*}) \simeq K_{*}$, concentrated in homological degree zero. 
\end{lemma}

\begin{proof}
Since $E_{*}$ is $2$-periodic and the above modules are even graded, it is enough to prove that $\varinjlim \Ext_{E_{0}}(E_{0} / \mfrak^{k}, K_{0}) \simeq K_{0}$, concentrated in homological degree zero. Since $E_{0}$ is a regular local ring, local duality implies that $\varinjlim \Ext_{E_{0}}^{i}(E_{0} / \mfrak^{k}, K_{0}) \simeq \Ext_{E_{0}}^{n-i}(K_{0}, E_{0})^{\vee}$, where by $(-)^{\vee}$ we denote the Matlis dual \cite{bruns1998cohen}. Because $K_{0} \simeq E_{0} / \mfrak$ is the unique simple $E_{0}$-module, it is Matlis self-dual and we deduce that it is enough to show that $\Ext_{E_{0}}(K_{0}, E_{0}) \simeq K_{0}$, concentrated in homological degree $n$. 

More generally, we claim that $\Ext_{E_{0}}(E_{0} / I_{k}, E_{0}) \simeq E_{0} / I_{k}$, concentrated in homological degree $k$, where $I_{k} = (p, u_{1}, \ldots, u_{k-1})$ for any $0 \leq k \leq n$. This is clear for $k = 0$ and the general case follows by induction from the long exact sequence of $\Ext$-groups associated to 

\begin{center}
$0 \rightarrow E_{0} / I_{k-1} \rightarrow E_{0} / I_{k-1} \rightarrow E_{0} / I_{k} \rightarrow 0$,
\end{center}
which ends the proof. 
\end{proof}

\begin{lemma}
\label{lemma:k_local_limits_of_finite_spectra_stable_under_smashing}
Let $X \simeq \varprojlim X_{i}$ be a limit diagram of $K$-local spectra such that $X_{i}$ and $X$ are $K$-locally dualizable. Then, for any $K$-local spectrum $Y$ we have $L_{K}(Y \wedge X) \simeq \varprojlim L_{K}(Y \wedge X_{i})$. 
\end{lemma}

\begin{proof}
Consider the collection of all $K$-local spectra $Y$ such that the needed condition holds. Since $X$ and $X_{i}$ are dualizable, smashing with them preserves $K$-local limits and we deduce that this collection is closed under limits. Since it also contains $S^{0}_{K}$ by assumption, we deduce that it is necessarily all of $\spectra_{K}$ by \cite[7.5]{hovey1999morava}.
\end{proof}
The following is the main result of this note. 

\begin{theorem}
\label{thm:picard_groups_of_k_local_cat_algebraic_at_large_primes}
Let $2p-2 > n^{2}+n$. Then, $E_{*}^{\vee}: \Pic(\spectra_{K}) \rightarrow \Pic(E_{*}^{\vee}E)$ is an isomorphism. 
\end{theorem}

\begin{proof}
If $2p-2 > n^{2}+n$, then $2p-2 \geq n^{2}$ and $(p-1) \nmid n$ and we've seen in \cref{prop:comparison_map_injective_at_large_primes} that under these conditions the homomorphism between Picard groups is injective. 

To verify surjectivity, we have to prove that if $M \in \Pic(E_{*}^{\vee}E)$, there exists a $K$-locally invertible spectrum $X$ with $E_{*}^{\vee} X \simeq M$. Observe that as an $E_{*}$-module, $M$ is necessarily free of rank one \cite[A.9]{hovey1999morava} and, without loss of generality, we can assume that it is even graded. Then, for each $k \geq 1$ we have $M / \mfrak^{k} M \simeq E_{*} / \mfrak^{k}$ as an $E_{*}$-module and so 

\begin{center}
$E_{*}^{\vee} E \widehat{\otimes}_{E_{*}} M / \mfrak^{k} M \simeq E_{*}E \widehat{\otimes} _{E_{*}} M / \mfrak^{k} M \simeq  E_{*}E \otimes _{E_{*}} M / \mfrak^{k} M$,
\end{center}
where the first isomorphism is \cite[A.7]{hovey1999morava} and the second follows from the fact that the last term is an $E_{*} / \mfrak^{k}$-module and so is already $L$-complete. Thus, we deduce that $M / \mfrak^{k} M$ is an $E_{*}E$-comodule in the usual, non-complete sense. 

Under the assumption $2p-2 > n^{2}+n$, in \cite[2.14]{pstragowski_chromatic_homotopy_algebraic} we construct the Bousfield splitting functor $\beta: \ComodE \rightarrow h \spectra_{E}$ valued in the homotopy category of $E$-local spectra with the property that $E_{*} \beta M \simeq M$ for any $M \in \ComodE$.

By construction, we have $E_{*} \beta (M / \mfrak^{k} M) \simeq M / \mfrak^{k} M$ and since the latter is $L$-complete, we deduce from \cite[3.2]{hovey2004some} that 
\[
E_{*}^{\vee} \beta(M / \mfrak^{k} M) \simeq E_{*} \beta (M / \mfrak^{k} M) \simeq M / \mfrak^{k} M.
\]
Note that we have a universal coefficient spectral sequence of signature
\begin{equation}
\label{equation:universal_coefficient_of_spectral_sequence}
\Ext_{E_{*}}(M/ \mfrak^{k} M, K_{*}) \Rightarrow K^{*} \beta (M / \mfrak^{k} M).
\end{equation}
As $E_{0}$ is a regular local ring of dimension $n$, this has a horizontal vanishing line at $s = n$ and in particular is strongly convergent; this will be important below. 

We let $X_{k} := L_{K} \beta(M / \mfrak^{k} M)$; by the above, this is a $K$-local spectrum with $E_{*}^{\vee} X_{k} \simeq M /\mfrak^{k} M$. In particular, $E^{\vee}_{*} X_{k}$ is degreewise finite and so $X_{k}$ is a finite $K$-local spectrum of type $n$ by  \cite[8.5]{hovey1999morava}.

As $\beta$ is a functor, we have maps $X_{k} \rightarrow X_{k-1}$ induced from the projections $M / \mfrak^{k} \rightarrow M / \mfrak^{k-1}$, well-defined up to homotopy, and we let $X := \varprojlim X_{k}$ denote the corresponding homotopy limit. Here, by the latter we mean that we pick a lift of the tower of $X_{k}$ to the $\infty$-category $\spectra_{K}$ and we compute the limit there. It is classical that up to equivalence the homotopy limit does not depend on the choice of that lift, since it can be defined using the triangulated structure alone. 

We first show that $X$ is invertible. Since $X_{k}$ are dualizable, we have $X \simeq D(\varinjlim DX_{k})$, where $D := F(-, S^{0}_{K})$ is the $K$-local Spanier-Whitehead dual and the colimit is the $K$-local one. Thus, it is enough to show that $\varinjlim DX_{k}$ is invertible; which we will verify by showing that
\[
K_{*} (\varinjlim DX_{k}) \simeq \varinjlim K_{*} (DX_{k}) \simeq K_{*}.
\]
As the universal coefficient spectral sequences of (\ref{equation:universal_coefficient_of_spectral_sequence}) have all the same horizontal vanishing line, by taking filtered colimits we obtain a strongly convergent spectral sequence of signature
\[
\varinjlim \Ext_{E_{*}}(M / \mfrak^{k} M, K_{*}) \Rightarrow K_{*} (\varinjlim DX_{k}) 
\]
Since $M$ is a free $E_{*}$-module of rank one, the needed statement follows from \cref{lemma:local_homology_of_residue_field}, and we deduce that $\varinjlim DX_{k}$, hence $X$, is invertible. 

Since $X_{k}$ and $X$ are $K$-locally dualizable, $L_{K}(E \wedge X) \simeq \varprojlim L_{K}(E \wedge X_{k})$ by \cref{lemma:k_local_limits_of_finite_spectra_stable_under_smashing}. After passing to homotopy groups, we obtain the Milnor exact sequence 
 
 \begin{center}
 $0 \rightarrow \varprojlim^{1} (M / \mfrak^{k}M)[-1] \rightarrow E_{*}^{\vee}X \rightarrow \varprojlim M / \mfrak^{k} M \rightarrow 0$
 \end{center}
 and since $M$ is free of rank one, the $\varprojlim^{1}$-term vanishes. We deduce that the second map must be an isomorphism, which ends the proof since $M \simeq \varprojlim M / \mfrak^{k} M$.
\end{proof}

\begin{rem}
\label{rem:collapse_of_descent_spectral_sequence_at_large_primes}
The following alternative argument, based on the descent spectral sequence, was pointed to us by Paul Goerss. The needed spectral sequence was not known at the time this article first appeared, but it has been since then constructed in the work of Heard \cite[\S 6C]{heard2021spkn}, giving an alternative proof of \cref{thm:picard_groups_of_k_local_cat_algebraic_at_large_primes}.

If $\ccat$ is a presentably symmetric monoidal $\infty$-category, then the Picard group can be lifted to the Picard \emph{space}, which we will denote by $\mathcal{P}ic(\ccat)$ \cite{mathew2016picard}. The latter is the $\infty$-groupoid of invertible objects in $\ccat$; it is an $\mathbb{E}_{\infty}$-space with multiplication induced from the tensor product.

The Picard group itself can be recovered through the relation $\Pic(\ccat) = \pi_{0} \mathcal{P}ic(\ccat)$. The higher homotopy groups of the Picard space are easy to describe, as we have $\pi_{t} \Picspace(\ccat) \simeq \pi_{t-1}\textnormal{aut}_{\ccat}(\monunit, \monunit)$ for $t > 0$, where $\monunit$ is the monoidal unit and $\textnormal{aut}_{\ccat}$ denotes the space of self-equivalences. 

By the work of Devinatz and Hopkins, the map $S^{0}_{K} \rightarrow E$ of commutative ring spectra is a $K(n)$-local pro-Galois extension in the sense of Rognes with Galois group $\mathbb{G}_{n}$ \cite{devinatz2004homotopy}, \cite{rognes2005galois}. In \cite{heard2021spkn}, Heard proves that associated to this extension we have a spectral sequence of signature\footnote{To be more precise, Heard does not identify all of the $E_{2}$-term with continuous group cohomology, but he identifies it in the range large enough ($(s, t) = (0, 0), (1, 1)$ or $s \geq 2$) for our argument to go through.}
\[
H^{s}_{c}(\mathbb{G}_{n}, \pi_{t} \Picspace (\Mod_{E})) \Rightarrow \pi_{t-s} \Picspace(\spectra_{K}),
\]
with differentials $d_{r}$ of degree $(r, r-1)$ and where the action of $\mathbb{G}_{n}$ on $\Mod_{E}$ is induced from that on $E$. This extends the previous work in the case of the finite Galois group \cite{mathew2016picard}, \cite{gepner_lawson_2021}. 

To get hold on the $E_{2}$-term, we need to understand the homotopy of the Picard space of $\Mod_{E}$, but this is not difficult. Since $E$ is even periodic and $E_{0}$ is regular local, any invertible $E$-module is free and so $\pi_{0} \Picspace (\Mod_{E}) \simeq \mathbb{Z}/2$ \cite{baker2005invertible}. Moreover, because $E$ is the monoidal unit of $\Mod_{E}$, we have $\pi_{1} \Picspace(\Mod_{E}) \simeq E_{0}^{\times}$ and $\pi_{t} \Picspace(\Mod_{E}) \simeq E_{t-1}$ for $t \geq 2$. 

If $(p-1) \nmid n$, then the Morava stabilizer group is of finite homological dimension $n^{2}$ and the $E_{2}$-term has a horizontal vanishing line. Furthermore, by standard considerations $H^{s}(\mathbb{G}_{n}, E_{t})$ vanishes unless $t$ is divisible by $2p-2$ \cite[4.2.1]{heard2015morava}. 

It follows that if $2p-2 \geq n^{2}$ and $(p-1) \nmid n$, then if drawn using the Adams grading, the $-1 \leq t-s \leq 1$ region of the above spectral sequence looks like 

\begin{center}
	\begin{tikzpicture}
		\node (TL) at (0, 1) {$ H_{c}^{1}(\mathbb{G}_{n}, \mathbb{Z}/2) $};
		\node (BL) at (0, 0) {$ 0 $}; 
		\node (TM) at (2.2, 1) {$  H_{c}^{1}(\mathbb{G}_{n}, E_{0}^{\times}) $};
		\node (BM) at (2.2, 0) {$ H_{c}^{0}(\mathbb{G}_{n}, \mathbb{Z}/2) $};
		\node (TR) at (4.4, 1) {$ 0 $};
		\node (BR) at (4.4, 0) {$ H_{c}^{0}(\mathbb{G}_{n}, E_{0}^{\times}) $};

		\node (ArrowStart) at (1.2, -0.8) {$ $};
		\node (ArrowEnd) at (3.1, -0.8) {$ $};
		
		\draw [->] (ArrowStart) to node[auto] {$t-s$} (ArrowEnd);
	\end{tikzpicture},
\end{center}
with only zeroes above. We deduce that in this range this spectral sequence collapses and yields a short exact sequence 

\begin{center}
$0 \rightarrow H_{c}^{1}(\mathbb{G}_{n}, E_{0}^{\times}) \rightarrow \pi_{0} \Picspace(\spectra_{K}) \rightarrow \mathbb{Z}/2 \rightarrow 0$. 
\end{center}
This means that the topological Picard group $\Pic(\spectra_{K}) \simeq \pi_{0} \Picspace(\spectra_{K})$ fits into a short exact sequence of the same form as the algebraic one, as explained in the introduction. One can then verify that $\Pic(\spectra_{K}) \rightarrow \Pic(E_{*}^{\vee} E)$ fits into a map of short exact sequences which is then an isomorphism by the five-lemma, giving a different proof of \cref{thm:picard_groups_of_k_local_cat_algebraic_at_large_primes}. 

In fact, the bound obtained in this way is slightly sharper, as one only needs $2p-2 \geq n^{2}$, rather than $2p-2 > n^{2}+n$. This comes from the fact that this argument avoids the use of the $E$-local category, since the homological dimension of $E_{*}E$ is $n^{2}+n$, while the homological dimension of $\mathbb{G}_{n}$ is just $n^{2}$. 

\end{rem}

\bibliographystyle{amsalpha}
\bibliography{chromatic_picard_groups_bibliography}

\providecommand{\bysame}{\leavevmode\hbox to3em{\hrulefill}\thinspace}
\providecommand{\MR}{\relax\ifhmode\unskip\space\fi MR }
\providecommand{\MRhref}[2]{%
  \href{http://www.ams.org/mathscinet-getitem?mr=#1}{#2}
}
\providecommand{\href}[2]{#2}
\begin{thebibliography}{GHMR14}

\bibitem[Bak09]{baker2009complete}
Andrew Baker, \emph{L-complete {Hopf} algebroids and their comodules},
  Contemporary Mathematics \textbf{504} (2009), 1.

\bibitem[BH98]{bruns1998cohen}
Winfried Bruns and H~J{\"u}rgen Herzog, \emph{{Cohen-Macaulay} rings},
  Cambridge University Press, 1998.

\bibitem[BH16]{barthel2016e2}
Tobias Barthel and Drew Heard, \emph{The {$E_{2}$}-term of the {$K(n)$}-local
  {$E_{n}$}-adams spectral sequence}, Topology and its Applications
  \textbf{206} (2016), 190--214.

\bibitem[BHN20]{barthel2020conjectures}
Tobias Barthel, Drew Heard, and Niko Naumann, \emph{On conjectures of
  {Hovey-Strickland} and {Chai}}, To appear in Selecta Mathematica (2020).

\bibitem[BR05]{baker2005invertible}
Andrew Baker and Birgit Richter, \emph{Invertible modules for
  commutative-algebras with residue fields}, Manuscripta mathematica
  \textbf{118} (2005), no.~1, 99--119.

\bibitem[DH04]{devinatz2004homotopy}
Ethan~S Devinatz and Michael~J Hopkins, \emph{Homotopy fixed point spectra for
  closed subgroups of the morava stabilizer groups}, Topology \textbf{43}
  (2004), no.~1, 1--47.

\bibitem[GH]{moduli_problems_for_structured_ring_spectra}
P.~G. Goerss and M.~J. Hopkins, \emph{Moduli problems for structured ring
  spectra},
  \url{http://www.math.northwestern.edu/~pgoerss/spectra/obstruct.pdf}.

\bibitem[GHMR05]{goerss2005resolution}
Paul Goerss, H-W Henn, Mark Mahowald, and Charles Rezk, \emph{A resolution of
  the {$K(2)$}-local sphere at the prime 3}, Annals of Mathematics (2005),
  777--822.

\bibitem[GHMR14]{goerss2014hopkins}
Paul Goerss, Hans-Werner Henn, Mark Mahowald, and Charles Rezk, \emph{On
  {Hopkins'} {Picard} groups for the prime 3 and chromatic level 2}, Journal of
  Topology \textbf{8} (2014), no.~1, 267--294.

\bibitem[GL21]{gepner_lawson_2021}
David Gepner and Tyler Lawson, \emph{Brauer groups and galois cohomology of
  commutative ring spectra}, Compositio Mathematica \textbf{157} (2021), no.~6,
  1211–1264.

\bibitem[Hea15]{heard2015morava}
Drew Heard, \emph{Morava modules and the {$K(n)$}-local {Picard} group},
  Bulletin of the Australian Mathematical Society \textbf{92} (2015), no.~1,
  171--172.

\bibitem[Hea21]{heard2021spkn}
\bysame, \emph{The $sp_{k, n}$-local stable homotopy category}, arXiv preprint
  arXiv:2108.02486 (2021).

\bibitem[HMS94]{hopkins1994constructions}
Michael~J Hopkins, Mark Mahowald, and Hal Sadofsky, \emph{Constructions of
  elements in {Picard} groups}, Contemporary Mathematics \textbf{158} (1994),
  89--89.

\bibitem[Hov04]{hovey2004some}
Mark Hovey, \emph{Some spectral sequences in {Morava} {$E$}-theory}, preprint
  \textbf{29} (2004).

\bibitem[HS99a]{hovey1999invertible}
Mark Hovey and Hal Sadofsky, \emph{Invertible spectra in the {$E(n)$}-local
  stable homotopy category}, Journal of the London Mathematical Society
  \textbf{60} (1999), no.~1, 284--302.

\bibitem[HS99b]{hovey1999morava}
Mark Hovey and Neil~P Strickland, \emph{Morava {$K$}-theories and
  localisation}, vol. 666, American Mathematical Soc., 1999.

\bibitem[LH]{lurie_hopkins_brauer_group}
Jacob Lurie and Michael Hopkins, \emph{On {Brauer} groups of {Lubin-Tate}
  spectra {I}}, http://www.math.harvard.edu/~lurie/papers/Brauer.pdf.

\bibitem[MS16]{mathew2016picard}
Akhil Mathew and Vesna Stojanoska, \emph{The {Picard} group of topological
  modular forms via descent theory}, Geometry \& Topology \textbf{20} (2016),
  no.~6, 3133--3217.

\bibitem[Pst21]{pstragowski_chromatic_homotopy_algebraic}
Piotr Pstr{\k{a}}gowski, \emph{Chromatic homotopy theory is algebraic when p>
  n2+ n+ 1}, Advances in Mathematics \textbf{391} (2021), 107958.

\bibitem[PV21]{pstrkagowski2021abstract}
Piotr Pstr{\k{a}}gowski and Paul VanKoughnett, \emph{Abstract goerss-hopkins
  theory}, Advances in Mathematics (2021), 108098.

\bibitem[Rog08]{rognes2005galois}
John Rognes, \emph{Galois extensions of structured ring spectra/stably
  dualizable groups: Stably dualizable groups}, vol. 192, American Mathematical
  Soc., 2008.

\bibitem[Str92]{strickland1992p}
NP~Strickland, \emph{On the p-adic interpolation of stable homotopy groups},
  Adams Memorial Symposium on algebraic topology, vol.~2, 1992, pp.~45--54.

\end{thebibliography}

\end{document}